\documentclass[12pt]{amsart}
\usepackage[utf8]{inputenc}
\usepackage[english]{babel}
\usepackage[all]{xy}
\usepackage{graphicx}
\usepackage{amsmath,amsthm,bm,mathrsfs,amscd,tikz}
\usepackage{amsfonts,amssymb,mathrsfs,mathtools}
\usepackage{verbatim}
\usepackage{float}
\usepackage{enumerate}
\RequirePackage{color}
\usepackage{xcolor}
\usepackage{soul}
\usetikzlibrary{calc}
\usepackage[colorlinks=true, pdfstartview=FitV, linkcolor=purple, citecolor=blue,
filecolor=violet, urlcolor=violet]{hyperref}
\usepackage{tikz-cd}

\usepackage[left=23mm, right=23mm, top=25mm, bottom=23mm]{geometry}

\theoremstyle{plain}
\newtheorem{theorem}{Theorem}[section]
\newtheorem{proposition}[theorem]{Proposition}
\newtheorem{corollary}[theorem]{Corollary}
\newtheorem{lemma}[theorem]{Lemma}
\newtheorem{conjecture}[theorem]{Conjecture}
\newtheorem{remark}[theorem]{Remark}

\theoremstyle{definition}
\newtheorem{definition}[theorem]{Definition}
\newtheorem{example}[theorem]{Example}

\newcommand{\Ext}{\operatorname{Ext}}

\newcommand{\add}{{\rm add}}

\newcommand{\End}{{\rm End}}

\newcommand{\SL}{{\rm SL}}
\newcommand{\CM}{{\rm CM}}
\newcommand{\Gr}{\rm Gr}

\newcommand{\Conf}{\rm Conf}

\newcommand{\tcfr}[3]{{\begin{psmallmatrix}  #1 & \\ & #3  \\ #2 & \end{psmallmatrix}}}
\newcommand{\A}{\mathcal{A}}

\title{Cluster Categorification of Rank 2 Webs}
\author{Ian Le and Emine Yıldırım}
\date{\today}

\address{Ian Le, Mathematical Sciences Institute
Australian National University, Canberra ACT 2601, Australia}
\email{ian.le@anu.edu.au}

\address{Emine Yıldırım, International Center for Mathematical Sciences - Sofia, Bulgarian Academy of Sciences, Sofia
1113, Bulgaria}
\email{e.yildirim@math.bas.bg}

\begin{document}

\maketitle

\begin{abstract} 
    The homogeneous coordinate ring of the Grassmannian $\Gr(k,n)$ has a well-known cluster structure~\cite{S}. There is a categorification of this cluster structure via a category of modules for a ring $\A_{k,n}$ due to Jensen-King-Su \cite{JKS}, building on work of Geiss-Leclerc-Schr\"oer \cite{GLS}, in which cluster variables correspond to indecomposable rigid modules. We give a combinatorial description of modules that correspond to rank $2$ cluster variables by using \emph{webs}. This conjecturally gives the categorification of all rank $2$ cluster variables.
\end{abstract}


\section{Introduction}

A central example of a cluster algebra comes from the homogeneous coordinate ring $\mathbb{C}[\widehat{\Gr (k, n)}]$ of the Grassmannians, which was first discussed by Fomin-Zelevinsky~\cite{FZ} for the case $k = 2$, and then by Scott~\cite{S} for $k \geq 2$. The study of these cluster algebras has led to many interesting connections between different parts of mathematics and physics.

One powerful technique in the study of cluster algebras is \emph{categorification}. Geiss-Leclerc-Schröer~\cite{GLS} and Jensen-King-Su~\cite{JKS} developed a representation-theoretic approach to the categorification of the cluster algebra structure on the Grassmannians. In this approach, cluster variables correspond to rigid indecomposable modules in some category of modules. Mutation of cluster variables corresponds to certain exact sequences of modules. Starting from a module, one can recover the corresponding element of the cluster algebra using the cluster character formula.

Another perspective on cluster structures on Grassmannians comes from invariant theory. Cluster monomials are elements in Lusztig’s dual semi-canonical basis. Fomin-Pylyavskyy~\cite{FP} study cluster monomials and more generally the dual semi-canonical basis through $SL_3$-invariant theory using the web basis introduced by Kuperberg~\cite{K}. They formulate precise and bold conjectures about expressing cluster variables in terms of webs. Their conjectures were one of the motivations behind this paper.

In this paper, we relate the perspectives of categorification and invariant theory. In particular, we study a family of rank $2$ functions on the Grassmannian. We express these functions via webs, and we show that they are cluster variables. Moreover, we give a combinatorial procedure for writing down the modules that categorify these webs. We conjecture that these are in fact all the rank $2$ cluster variables on the Grassmannian, and we provide evidence for this conjecture. Our bijection between functions and modules is a natural generalization of the rank $1$ version given in \cite{JKS}.

Let us outline the contents of this paper. In Sections~\ref{sec:web2} and~\ref{sec:mod2}, we give a brief introduction to rank $2$ webs and modules we work with, introduce the definitions and necessary notation. In Section~\ref{sec:wm}, we construct the bijection between webs and modules via the map $\Psi$. We also state our main theorem (Theorem~\ref{mainthm}). An important conceptual ingredient of the paper is given in Section~\ref{sec:stretching}, where we describe certain operations on webs and modules that we call ``stretching''. These operations are called \emph{stretching functions}. Section~\ref{sec:rank2webs} is devoted to proving that the webs we consider are all cluster variables. Finally, in Section~\ref{sec:cactusmain}, we begin by recalling some terminology and results from~\cite{L19} about the cactus sequence. We explicitly write down short exact sequences of modules for the mutations in this sequence, which allows us to prove the correspondence between rank $2$ webs and modules. 

\subsection*{Acknowledgements} We thank Alastair King for his valuable comments on the first draft of this paper and pointing out the use of immanants (Remark~\ref{rem:immanants}). This project started during the Junior Trimester Programme: New Trends in Representation Theory at the Hausdorff Institute for Mathematics in Bonn, 2020. Part of this work is completed at the CAR programme at the Isaac Newton Institute for Mathematical Sciences in 2021 (supported by EPSRC grant no EP/R014604/1). We thank organisers and both institute for their support. EY is supported by different bodies: in part by the Royal Society Wolfson Award RSWF/R1/180004 and by the Ministry of Education and Science of the Republic of Bulgaria with grant no {DO1-239/10.12.2024} and the Simons Foundation with grant no {SFI-MPS-T-Institutes-00007697}.

\section{Webs of Rank $2$}~\label{sec:web2}

A $k$-web is a bipartite graph drawn in a disc with weights on the edges such that the sum of weights around each interior vertex is $k$. We will be interested in $k$-webs that are inside a disc with $n$ vertices on the boundary, all colored black, numbered $1$ to $n$. Denote  $[n]:=\{1, 2,\cdots, n\}.$ 

Any $k$-web will have some number of edges ending at the boundary vertices. Let us call such edges \emph{leaves}. Note that leaves can intersect each other inside the disk or on the boundary vertices. We will restrict our attention to webs such that the leaves all have weight $1$. A boundary vertex can have multiple (or no) leaves attached to it.

For any $k$-web with all leaves of weight $1$, the number of leaves is a multiple of $k$. We will call a $k$-web with $rk$ leaves a $k$-web (or simply a web) of rank $r$.

In this paper, we will only consider rank $1$ and rank $2$ webs. Rank $1$ webs consist of one interior white vertex and $k$ boundary black vertices. The rank $2$ webs $W$  we consider will consist of a tree with one interior black vertex and three white vertices inside the disk and edges as shown below: White vertices are connected to the boundary vertices with a finite number of edges. 

\begin{figure}[H]
    \centering
    \includegraphics[width=9cm]{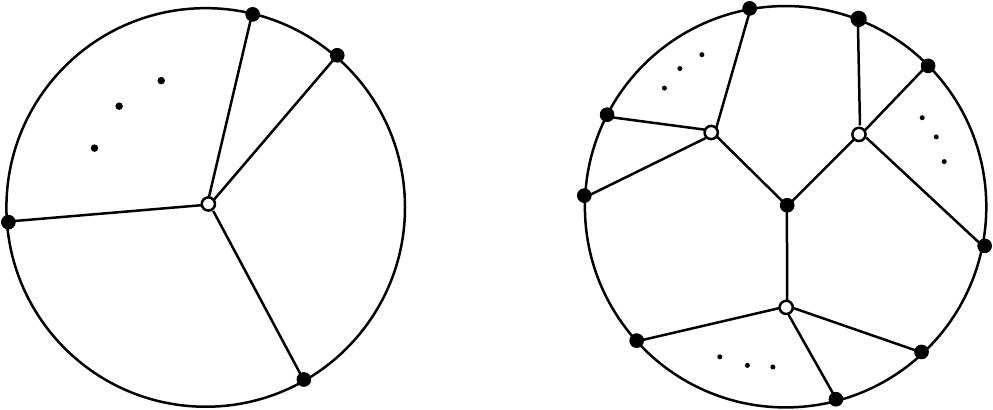}
    \caption{An illustration of rank $1$ and $2$ webs.}
    \label{rank2web}
\end{figure}

The set $[n]$ has a natural cyclic structure when placed on a boundary of a circle. For any $i \in [n]$ the cyclic order induces a total order $<_i$, where $i <_i i+1 <_i \cdots <_i n <_i 1 <_i \cdots <_i i-1$. Let $R, S, T$ be three subsets of $[n]$. We will say that they are \emph{cyclically separated} if there is some $i$ such that all the elements of $R$ are less than all the elements of $S$, which are in turn less than all the elements of $T$ in the ordering $<_i$. In particular, the cyclically separated sets are disjoint.

Let $R, S, T$ be cyclically separated subsets of $[n]$ and let $V \subset [n]$ be an arbitrary (possibly empty) set disjoint from $R, S$ and $T$. Suppose further that $|R|=a, |S|=b, |T|=c, |V|=d$ and $a+b+c+2d=2k$. Then we can define a rank $2$ web $W(R,S,T,V)$ as follows: The elements in $R, S$, $T$ and $V$ will be the labels of boundary vertices that are leaves of the tree. The vertices in $R$ will be attached to the first white vertex; the vertices in $S$ will be attached to the second white vertex, and the vertices of $T$ will be attached to the third white vertex. Each vertex of $V$ will be attached to two white vertices, but which white vertices they are attached to inside the disk does not matter. Changing the attachments of the vertices in $V$ will give us equivalent webs, in the sense that they give the same function on the Grassmannian, see~\cite[Section 5]{FP}.

Denote by $\mathcal{W}$ the set of such rank $2$ webs. 

A $k$-web of rank $r$ on the vertex set $[n]$ induces a degree $r$ function on the homogeneous coordinate ring, $\mathbb{C}[\Gr(k,n)]$, of the Grassmannian $\Gr(k,n)$: Such a web gives a function of $n$ vectors in a $k$-dimensional vector space; moreover, this function is invariant under the action of $SL_k$. Therefore, these invariant functions in $\mathbb{C}[\Gr(k,n)]$ can be understood by viewing a point in the Grassmannian as a $k \times n$ matrix and evaluating the invariant on the columns of the matrix.

For example, recall that a Plücker coordinate in $\mathbb{C}[\Gr(k,n)]$ is an $\SL_k$- invariant function of $k$ vectors. This is represented by $k$-valent graph shown in Figure~\ref{rank2web} (left). Throughout this paper, we will therefore view webs as functions on the Grassmannian. For a detailed treatment of the subject, we refer the reader to ~\cite[Example 3.4]{FLL},~\cite[Section 5]{FP}.

\begin{example}
    Consider the homogeneous coordinate ring $\mathbb{C}[\Gr(3,6)]$. The cluster algebra structure found in this ring is of Dynkin type $D_4$; twenty-two cluster variables of which we have six frozen variables. Since all Plücker coordinates are cluster variables, we have $\binom{6}{3}=20$ of them. This means we have two non-Plücker cluster variables in this case. They can be found using exchange relations, or skein relations on webs~\cite{FP}.

\begin{figure}[H]
    \centering
    \includegraphics[width=13cm]{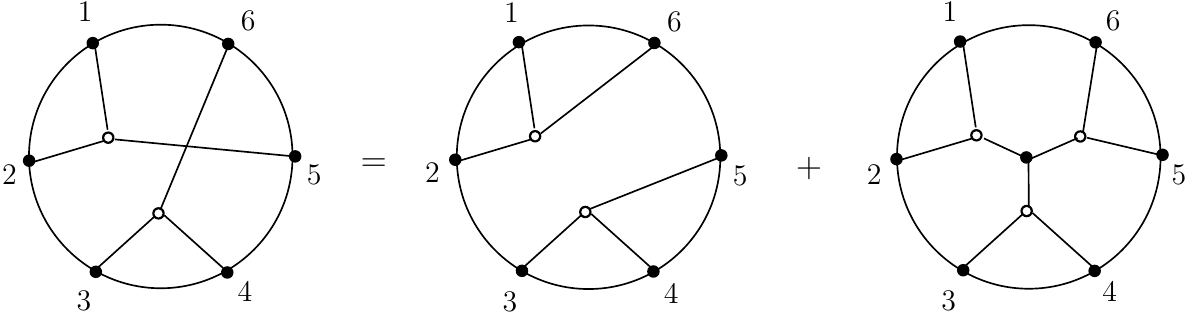}
    \caption{A rank $2$ web in an exchange relation.}
    \label{D4}
\end{figure}

The rightmost web in Figure~\ref{D4} and its rotation by one are the two rank $2$ webs which are cluster variables on $\Gr(3,6)$.
\end{example}

\begin{conjecture} The set $\mathcal{W}$ realizes the complete set of rank $2$ cluster variables on the Grassmannian $\Gr(k,n)$.
\end{conjecture} 

Strong evidence for this is the work of Fraser \cite{F22} which describes the full canonical basis of rank $2$ functions as webs. Among those webs, the only arborizable ones are those in $\mathcal{W}$. Informally speaking, arborization is a process that turns a planar web into a possibly self-intersecting \emph{tree diagram} if possible, see~\cite[Definition 10.3]{FP}. It is believed that a web must be arborizable in order to be a cluster variable \cite{FP}. (They conjecture this for $k=3$, but we believe it has a good chance of being true for $k>3$ as well.) Note that the number of webs in $\mathcal{W}$ is also precisely the number of isomorphism classes of rank $2$ indecomposable rigid modules predicted in \cite{BBL}.

\section{JKS Modules}~\label{sec:mod2}

We follow the notation of Jensen-King-Su~\cite{JKS}. Their paper gives an additive categorification of the cluster algebra structure in $\mathbb{C}[\Gr (k,n)]$ by using certain Cohen-Macaulay (CM) modules. We will be interested in rank $1$ and rank $2$ CM modules defined in~\cite{JKS}. 
Let us give a gentle introduction to this categorification. Fix $k, n.$ Let $Q=(Q_0,Q_1)$ be a quiver with a vertex set $Q_0$ and an edge set $Q_1$ as shown in Figure~\ref{quiverQ}. Let $\A_{k,n}$ be the path algebra of $Q$ with the relations `$y_{i+1}x_{i+1}=x_iy_i$' and `$x_{i+k}\cdots x_{i+2} x_{i+1}=y_{i-(n-k-1)} \cdots y_{i-1} y_i$' for each vertex $i$. We consider $1,\cdots, n$ cyclically, so $0=n$ and so on. Notice that we have $2n$ relations in total. 

\begin{figure}[ht]
    \centering
    \includegraphics[width=5cm]{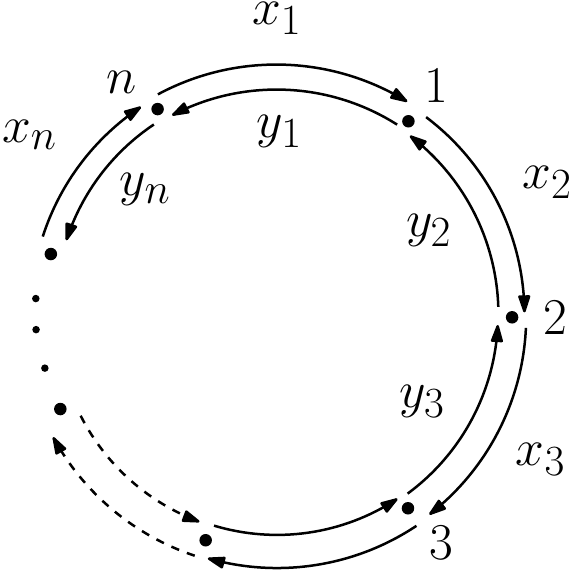}
    \caption{The quiver $Q.$}
    \label{quiverQ}
\end{figure}

For the simplicity of the presentation, we mostly draw $Q$ in a line with $n+1$ vertices thinking that vertices $0$ and $n$ identified (see Figure~\ref{1ex}).

In \cite{JKS}, they showed that rank $1$ modules are in bijection with the $k$-subsets of $[n]$ which are Plücker coordinates in the corresponding Grassmannians. We use a \emph{profile}, defined in~\cite{JKS}, for rank $1$ modules as below. 

\begin{definition} A \emph{contour} of a rank $1$ module is a series of up and down moves written respecting $k$-subsets of $[n]$ as down steps and the rest of $n-k$ as up steps.

A \emph{profile} of a module is a collection of contours.
\end{definition}

\begin{figure}[H]
    \centering
    \includegraphics[width=3.7cm]{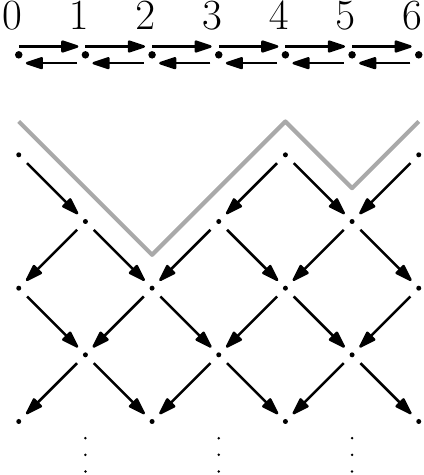}
    \caption{The quiver Q and profile of a rank $1$ module $I=\{1,2,5\}$. The contour of I is shaded in gray.}
    \label{1ex}
\end{figure}

In Figure~\ref{1ex}, we consider a rank $1$ module $I=\{1,2,5\}$ for $Gr(3,6)$. We describe the contours of rank $1$ modules as a sequence of $n-k$ $U$'s (up) and $k$ $D's$ (down). The contour of $I=\{1,2,5\}$ in $Gr(3,6)$ is $DDUUDU.$

In \cite{JKS} they further show that every rigid indecomposable CM module has a generic filtration by rank $1$ modules. Thus we can write a profile for $M$ by superposing the profiles of rank $1$ modules in this filtration with an alignment that respects the order of the rank 1 modules in the filtration; see example below.


We will be working with the rank $2$ modules $M$ which are formed by a packing of two rank $1$ modules, denoted by $\frac{I}{J}$. We call $\frac{I}{J}$ a profile of $M$ as in~\cite{JKS}. We draw the contour of $I$ first and then add the contour of $J$ as close as possible to $I$. We are ``packing'' the contours as close together as possible -- this is necessary in order to obtain indecomposable rank $2$ modules. Examples are shown below in Figure~\ref{2ex}.

Note that there are regions appearing between the contours in a rank $2$ module. We will call these regions ``boxes'' (a similar definition was given in \cite[Definition 4.1]{BBEL}).

Let us make this notion precise. We think of a profile as consisting of a string of length $n$ containing the letters $\frac{U}{U}$, $\frac{U}{D}$, $\frac{D}{U}$, $\frac{D}{D}$. A box of size $\alpha$ is a string consisting of $(\frac{U}{D})^\alpha(\frac{D}{U})^\alpha$ with $(\frac{U}{U})$ and $(\frac{D}{D})$ interspersed freely, with the restriction that it begins with $\frac{U}{D}$ and ends with $\frac{D}{U}$.

We are interested in profiles that consist of some cyclic arrangement of three boxes with strings of $(\frac{U}{U})$ and $(\frac{D}{D})$ interspersed among them.  For instance, a profile $\frac{I}{J}$ could be $\frac{DUUDUDUU}{UDUUDUUD}.$ which is the first profile in Figure~\ref{2ex}. We will see below that there is a unique indecomposable module with such a profile. 

Denote the set of all such rank $2$ modules by $\mathcal{M}$. 
 
\begin{figure}[ht]
    \centering
    \includegraphics[width=12.5cm]{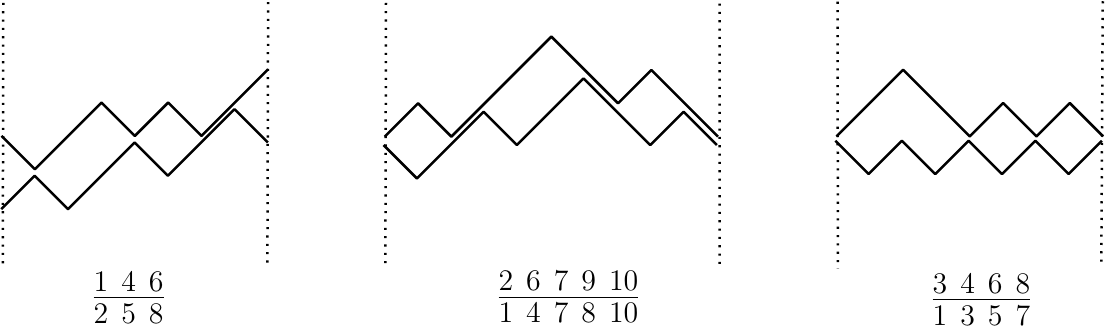}
    \caption{Examples of rank $2$ modules with three boxes.}
    \label{2ex}
\end{figure}

Let us call $A, B$ and $C$ three boxes in the profile of $M$ that are divided by contours of $I$ and $J$. Let us think of each box as consisting of a consecutive cyclic substring of $1, 2, \dots, n$. Denote $M'=\frac{I'}{J'}=\frac{I\setminus{I\cap J}}{J\setminus{I\cap J}}$ and set $I\cap J=V.$ This is to remove any common label in the profiles of $I$ and $J$ so that $I'\cap J'=\emptyset.$ Set also $I_A:=I'\cap A$, $I_B:=I'\cap B$ and $I_C:=I'\cap C$. Define $J_A, J_B$ and $J_C$ respectively.

\begin{lemma}~\label{indecs} Let $\frac{I}{J}$ be a profile consisting of three boxes with $(\frac{U}{U})$ and $(\frac{D}{D})$ interspersed. Then there is a unique indecomposable module of rank $2$ with this profile. All the modules in $\mathcal{M}$ are of this type.
\end{lemma}
   
\begin{proof}

All the profiles we are considering have the following poset:
           \[\xymatrix@C=1em@R=1em{ \Bbbk  \ar^{}[dr] & \Bbbk  \ar[d] & \Bbbk  \ar[dl] \\
           & \Bbbk^2   & }\] 

Thus we are looking at three lines in $\Bbbk^2$. There is a unique, up to isomorphism, indecomposable representation of this quiver with the given dimension vectors. Up to changes of basis it is given by three different vector subspaces of dimension one, which embed into a two-dimensional vector space with maps 
$\begin{bmatrix}
               1 \\ 
               0
           \end{bmatrix}, \begin{bmatrix}
               0 \\ 1
           \end{bmatrix}, \begin{bmatrix}
               1 \\ 1
           \end{bmatrix}$ 
respectively. The three boxes we have in the profile are associated to the three different one-dimensional vector spaces.
\end{proof}

An alternative proof can be found in~\cite[Theorem 1.2]{BBL}. 

A module $M$ is called \emph{rigid} if and only if $\Ext^1(M,M)=0$. It is conjectured in~\cite[Conjecture 5.15]{BBE} that there are $2\binom{n}{6}\binom{n-6}{k-3}$ rigid indecomposable rank $2$ modules corresponding to real roots. In the same paper it is shown that $\CM(A_{k,n})$ has at most $2\binom{n}{6}\binom{n-6}{k-3}$ rigid indecomposable rank $2$ modules which correspond to real roots. In~\cite[Theorem 4.7]{BBEL}, they show that there are at least

\begin{equation}~\label{counting-mods}
N_{k,n} := \sum^k_{r=3}(\frac{2r}{3}p_1(r)+2r p_2(r)+4r p_3(r))\binom{n}{2r}\binom{n-2r}{k-r} 
\end{equation}
rigid rank $2$ indecomposable modules that correspond to roots and conjecture that this is all of them. In this formula, $p_1(r), p_2(r), p_3(r)$ are the number of partitions of $r$ into three parts which are equal, two of which are equal, or all of which are distinct, respectively. This is precisely the number of modules that we constructed. 

We also know from~\cite[Proposition 5.7]{BBE} that if $M$ is an indecomposable rank $2$ module with profile $\frac{I}{J}$ and $M$ corresponds to a real root, then $|I \cap J|=k-3$. By stretching the modules, Definition~\ref{def:stretching}, corresponding to the web in Figure~\ref{D4}, we obtain $2\binom{n}{6}\binom{n-6}{k-3}$ modules corresponding to real roots, verifying~\cite[Conjecture 5.15]{BBE}. This implies that indecomposable rank $2$ modules of $\mathcal{M}$ are rigid.

\begin{remark}~\label{rem:immanants}
    In our notation for webs, the partition corresponds to $r=a+b+c,$ where $a=|R|$, $b=|S|, c=|T|, k-r=|V|$ for a web $W(R,S,T,V).$

Note that by using immanants \cite{F22}, we can also write 

\begin{equation}
N_{k,n} := \sum^k_{r=3}2\binom{r}{3}\binom{n}{2r}\binom{n-2r}{k-r}.
\end{equation}

Here, $2\binom{r}{3}$ counts the ways of dividing $2r$ cyclically arranged elements into three sectors of even size, which gives via the immanant construction the webs in our set $\mathcal{W}.$ Note that these sectors have sizes $2a, 2b, 2c$ for a web $W(R,S,T,V)$ in the notation above.
\end{remark}

\section{Webs and Modules}~\label{sec:wm}

We give the recipe for translating webs into rank $2$ modules. Let us define a function 

\[\Psi : \mathcal{W}\mapsto \mathcal{M}\] from the set of rank $2$ webs $\mathcal{W}$ to the set of indecomposable rank $2$ modules $\mathcal{M}$ as follows.

In a rank $2$ web, we have three white vertices inside the disk and each white vertex is connected to several vertices on the boundary. As before, on the boundary, there will be $a$ vertices connected to the first white vertex, $b$ vertices connected to the second white vertex, and $c$ vertices connected to the third white vertex. There will be $d$ vertices that are attached to any two white vertex. Set $a=\alpha+\beta$, $b=\beta+\gamma$ and $c=\alpha+\gamma$ and solve for $\alpha, \beta$ and $\gamma$. Set 

\begin{enumerate}[(i)]
    \item $R_1=\{r_1,r_2,\cdots,r_{\alpha}\}$, $R_2=\{r_{\alpha+1},\cdots,r_{\alpha+\beta}\}$,
    \item $S_1=\{s_1,s_2,\cdots,s_{\beta}\}$, $S_2=\{s_{\beta+1},\cdots,s_{\beta+\gamma}\}$,
    \item $T_1=\{t_1,t_2,\cdots,t_{\gamma}\}$, and $T_2=\{t_{\gamma+1},\cdots,t_{\gamma+\alpha}\}.$
\end{enumerate}   
Here the vertices of $R$, $S$ and $T$ are listed cyclically.

Let $W=W(R,S,T,V)=W(R_1\cup R_2,S_1\cup S_2,T_1\cup T_2,V)$. Note that the $d=|V|$ and the vertices of $V$ may sit anywhere on the boundary.

Then, 

\[\Psi(W)=\frac{\mathbf{r_1}\mathbf{s_1}\mathbf{t_1}\mathbf{v}}{\mathbf{r_2}\mathbf{s_2}\mathbf{t_2}\mathbf{v}}\]

where $\mathbf{r_1}$ is just the listing of the elements of $R_1$ as $r_1r_2\cdots r_{\alpha}$, and similarly for the others. 

The resulting rank $2$ module $\frac{\mathbf{r_1}\mathbf{s_1}\mathbf{t_1}\mathbf{v}}{\mathbf{r_2}\mathbf{s_2}\mathbf{t_2}\mathbf{v}}$ consists of three boxes: $R_2 \cup S_1$, $S_2 \cup T_1$, $T_2 \cup R_1$. These are interspersed with $U/U$ steps (those indices not contained in $R, S, T$ or $V$) and $D/D$ steps (those indices contained in $V$).

Assume $d=0$, then we can visualize this map as follows.

\begin{figure}[ht]
    \centering
    \includegraphics[width=13cm]{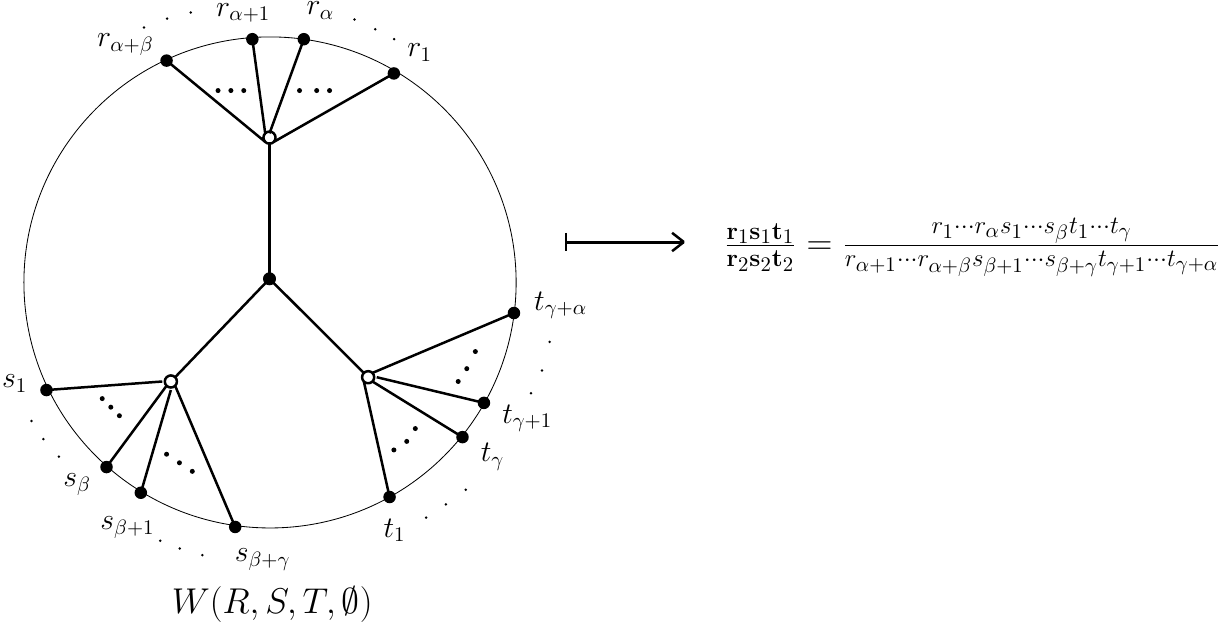}
    \caption{An illustration of a web and the corresponding module.}
    \label{web-mod}
\end{figure}

It is immediate from the definition of the function $\Psi$ that the image of $\Psi$ consists precisely of those rank $2$ modules which are three boxes (of sizes $\alpha, \beta, \gamma$) interspersed with $U/U$ and $D/D$. More precisely, the image of $\Psi$ is an indecomposable rank $2$ module.

The map $\Psi$ is well-defined, surjective and injective; an indecomposable rank $2$ module $M$ which is formed by three boxes is uniquely determined by its profile by Lemma~\ref{indecs}.

\begin{example} We provide a list of examples here.
\begin{figure}[H]
    \centering
    \includegraphics[width=9cm]{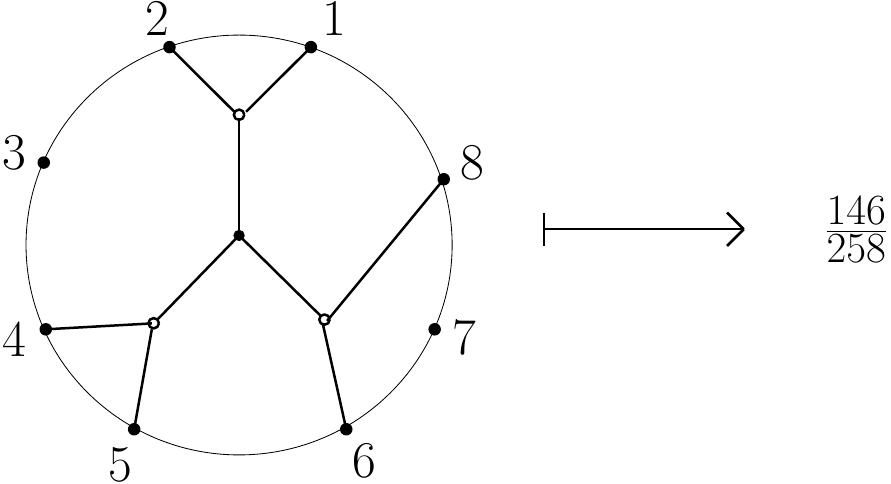}  
\end{figure}
\begin{figure}[H]
    \includegraphics[width=9cm]{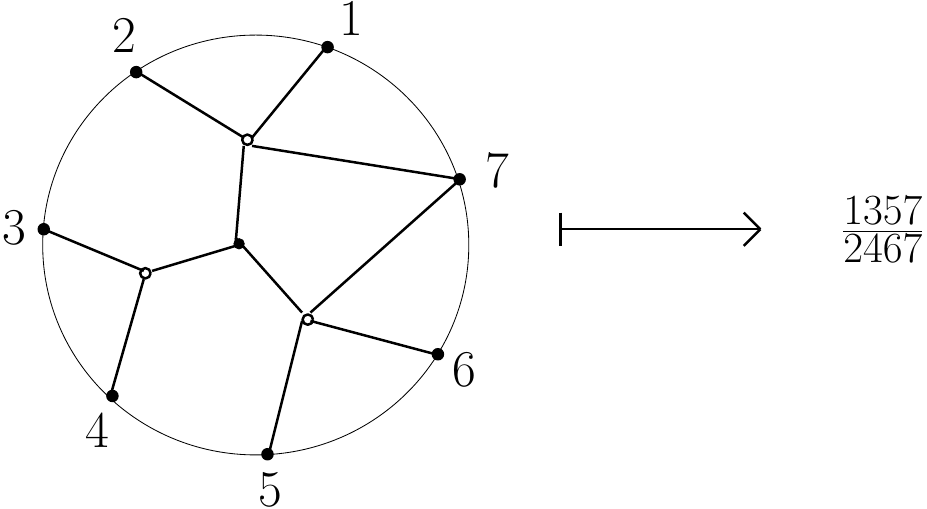}
\end{figure}
\begin{figure}[H]
    \includegraphics[width=9cm]{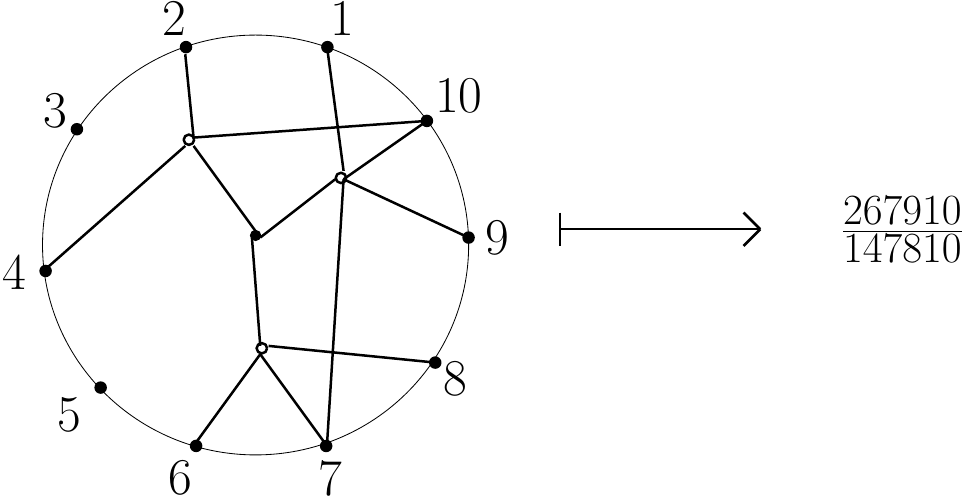}
\end{figure}
\end{example}

\begin{example}
Consider $\Gr(5,12)$. By Equation~(\ref{counting-mods}), there should be $N_{5,12}$ rigid rank $2$ indecomposable modules.

\[N_{5,12} = 2 \binom{12}{6}\binom{6}{2}+8\binom{12}{8}\binom{4}{1}+20\binom{12}{10}\binom{2}{0}.\]

Note that the count of the modules corresponding to real roots is $2 \binom{12}{6}\binom{6}{2},$ corresponding to the first term in the sum.

Now, let us count the rank $2$ webs in $\mathcal{W}$. We can analyze the possible values of $a,b,c,d$. (Recall that $a+b+c+2d=10$, and $a, b, c$ must satisfy the triangle inequality.) The first terms counts those webs with $a=b=c=d=2$. The second counts those with $a=b=3, c=2, d=1$. The third counts those with $a=4, b=c=3, d=0$ or $a=b=4, c=2, d=0$.
\end{example}

We now state the main theorem of the paper.

\begin{theorem}\label{mainthm}
If $W \in \mathcal{W}$, then $W$ is a cluster variable, and the corresponding module is $\Psi(W) \in \mathcal{M}.$ In particular, the cluster character of $\Psi(W)$ is $W.$
\end{theorem}

The first part of the theorem is proven in Section~\ref{sec:rank2webs}. The second part is proved in Section~\ref{sec:cactusmain}.

\section{Stretching Functions}~\label{sec:stretching}

In this section, we will define the operations of \emph{stretching} (up and down) on webs and modules. These operations will give us ability to work with functions across different Grassmannians. 

We begin with the stretching of modules. As a first step, consider the following function $s_i: [n] \rightarrow [n+1]$:

\[s_i(j)=\left\{
  \begin{array}{lr}
    j & : j < i\\
    j+1 & : j \ge i
  \end{array}
\right.\]

\begin{definition}~\label{def:stretching}
Suppose $I \subset [n]$ with $|I|=k$, so that $I$ defines a rank $1$ module.
    \begin{enumerate}
        \item The \textbf{up stretching function} $U_i$ takes $I$ to $U_i(I)$, where $U_i(I)$ is the result of applying the map $s_i$ to all the indices of $I$. In other words, $j\in I \iff s_i(j) \in U_i(I)$ and $i \notin U_i(I)$.
        \item We also define the \textbf{down stretching function}: $D_i(I)=U_i(I)\cup \{i\}.$ 
    \end{enumerate}
\end{definition}

We would like to define stretching functions on modules. We can first define stretching on profiles by just applying stretching to each contour in the profile. For example, $U_i(\frac{I}{J}) = \frac{U_i(I)}{U_i(J)}$.

We have described the effect of $U_i$ and $D_i$ on the contours of a module. It is simple to see how to extend $U_i$ and $D_i$ to functors on categories of modules:

\begin{definition}
Consider a module $M$. Then $U_i(M)$ is the module with $(U_i(M))_j = M_j$ for $j \leq i$ and $(U_i(M))_j = M_{j-1}$ for $j > i$. The maps $x$ and $y$ are defined as before, except that $y_i$ is the identity map on $M_i$, and $x_i$ is the map given by multiplication by $t$. This is determined by the relation $y_{i+1}x_{i+1}=t=x_{i}y_{i}.$

Similarly, we define $D_i(M)$ by $(D_i(M))_j = M_j$ for $j \leq i$ and $(D_i(M))_j = M_{j-1}$ for $j > i$, but $x_i$ is the identity map, and $y_i$ is given by multiplication by $t$.    
\end{definition}

\begin{remark} Stretching was also defined in \cite{BBEL}, where $U$ and $D$ were called ``increase'' and ``co-increase.''
\end{remark}

\begin{remark} It is clear that stretching preserves indecomposability of modules. It is less obvious that it preserves rigidity and reachability. We will show later in Lemma~\ref{stretchinglemma} that stretching a reachable rigid module gives another reachable rigid module.
\end{remark}

Let us now define these up and down stretching functions for webs of rank $2$.

\begin{definition}
    \begin{enumerate}
        \item $U_i(W)=W(U_i(R), U_i(S), U_i(T), U_i(V))$ where $U_i(R), U_i(S), U_i(T), U_i(V) \subset [n+1]$.
        \item $D_i(W)=W(U_i(R), U_i(S), U_i(T), D_i(V))$ where $U_i(R), U_i(S), U_i(T), D_i(V) \subset [n+1]$.
\end{enumerate}
\end{definition}

\begin{example} Here is an example of stretching of a rank $2$ web and a rank $2$ module under the up stretching map $U_3$.

\begin{figure}[ht]
    \centering
    \includegraphics[width=10.7cm]{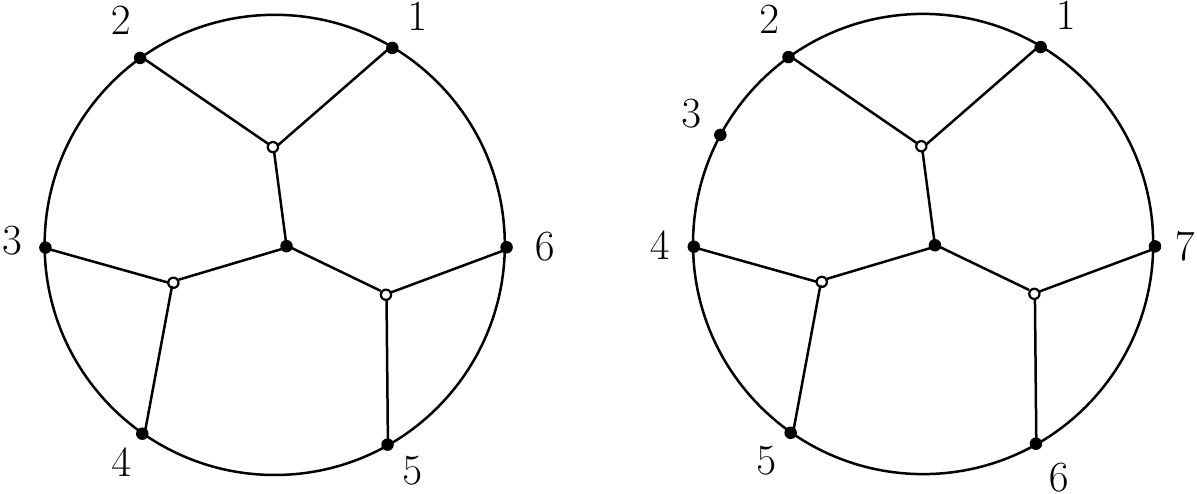}
    \caption{Stretching of a rank $2$ web by $U_3$.}
    \label{ger}
\end{figure}

\[W(\{1,2\},\{3,4\},\{5,6\},\emptyset) \longmapsto W(\{1,2\},\{4,5\},\{6,7\},\emptyset)\]

\[\frac{135}{246}\longmapsto U_3(\frac{135}{246})=\frac{U_3(135)}{U_3(246)}=\frac{146}{257}\]

\end{example}

Stretching comes from natural maps; 
$$F_i: Gr(k,n+1) \rightarrow Gr(k,n),$$
$$G_i: Gr(k+1,n+1) \rightarrow Gr(k,n).$$
If we view a point in $Gr(k,n+1)$ as a $k \times (n+1)$ matrix, the map $F_i$ is just forgetting the $i$-th column. The map $G_i$ is the composition

$$Gr(k+1,n+1) \simeq Gr(n-k,n+1) \xrightarrow{F_i} Gr(n-k,n) \simeq Gr(k,n),$$
where the isomorphisms come from duality.

We will see in Lemma~\ref{stretchinglemma} that the stretching maps $U_i$ (respectively $D_i$) are induced by pullback of functions under $F_i$ (respectively $G_i$).

The maps $F_i$ and $G_i$ are \emph{cluster fibrations}; see~\cite[Proposition 3.4 (1)]{FL19}. This means that any cluster variable on $Gr(k,n)$ pulls back under $F_i^*$ (respectively $G_i^*$) to a cluster variable on $Gr(k,n+1)$ (respectively $Gr(k+1,n+1)$).

In fact, more is true. Let $C$ be a cluster on $Gr(k,n)$ consisting of Pl\"ucker coordinates. Then pulling back under $F_i^*$ gives cluster variables on $Gr(k,n+1)$ which can be completed to a cluster $\tilde{C}$ on $Gr(k,n+1)$. Note that the variables that are added are all frozen variables, and all the arrows in the quiver for $\tilde{C}$ only connect these frozen variables with variables that were frozen in $C$. If we start mutating the cluster $C$, there is a corresponding set of mutations on the cluster $\tilde{C}$. These mutations never touch the variables frozen in $C$, so the arrows involving the new frozen variables in $\tilde{C}$ never change. Thus every cluster on $Gr(k,n)$ can be completed to one on $Gr(k,n+1)$, and the quivers for clusters on $Gr(k,n)$ are all subquivers of the corresponding clusters on $Gr(k,n+1)$. To summarize,

\begin{proposition}~\cite[Section 3]{FL19} The following is true for any cluster $C$ on $Gr(k,n)$:
\begin{itemize}
    \item[(i)] Pulling back the cluster variables in $C$ via $F_i^*$ (respectively $G_i^*$) gives cluster variables on $Gr(k,n+1)$ (respectively $Gr(k+1,n+1)$) that can be completed to a cluster $\tilde{C}$ on $Gr(k,n+1)$ (respectively $Gr(k+1,n+1)$).
    \item[(ii)] $C$ is a subquiver of $\tilde{C}$.
    \item[(iii)] The additional arrows in $\tilde{C}$ only connect frozen vertices in $\tilde{C}$ to arrows that were frozen in $C$. Thus the additional arrows are independent of the cluster $C$.    
\end{itemize}
\end{proposition}

Note that the reachable rigid indecomposable modules are in one-to-one correspondence with the cluster variables under the cluster character map~\cite{GLS, JKS}. Moreover, mutation of cluster variables corresponds to mutation of modules. We give below the properties that characterize the mutation of modules. The following lemma shows that modules corresponding to cluster variables behave well under stretching.

\begin{lemma}\label{stretchinglemma} Suppose that $M$ is an indecomposable, reachable rigid module corresponding to a cluster variable $f$. Then $U_i(M)$ and $D_i(M)$ for each $i\in [n]$ are indecomposable, reachable rigid modules corresponding to the functions $F_i^*(f)$ and $G_i^*(f)$, respectively.
\end{lemma}

\begin{proof}~\label{proof:behavewell}
The proofs for $U_i$ and $D_i$ are similar, so let us treat $U_i$.

Say that a cluster $C$ \textit{behaves well} under stretching if the following hold:
\begin{enumerate}
    \item If $f \in C$ corresponds to a module $M$, then $F_i^*(f)$ corresponds to the module $U_i(M)$
    \item Let $f$ and $g$ in $C$ correspond to modules $M$ and $N$. Then any arrow in the quiver between the vertices of $f$ and $g$ gives a corresponding irreducible map $\alpha: M \rightarrow N$. In this case, we have that $U_i(\alpha): U_i(M) \rightarrow U_i(N)$ is also an irreducible map.
\end{enumerate}

Start with a cluster $C$ for $Gr(k,n)$ consisting of Pl{\"u}cker coordinates, for example, any cluster coming from a plabic graph. 

For $f \in C$, we know that if $M$ is the module corresponding to $f$, then $U_i(M)$ is the module corresponding to $F_i^*(f)$. This is because $f$ was chosen to be a Pl{\"u}cker coordinate.

Moreover, one can check that the irreducible maps between the modules corresponding to cluster variables in $C$ give irreducible maps after stretching. Thus the initial cluster $C$ behaves well under stretching. 

We wish to show that the property of behaving well under stretching is stable under mutation:

\begin{lemma}~\label{lm:strc} Suppose $C$ is a cluster for $Gr(k,n)$ which behaves well under stretching as in Proof \ref{proof:behavewell}. If $f'$ is the mutation of $f$ having corresponding module $M'$, then $F_i^*(f')$ corresponds to the module $U_i(M')$. Moreover, the irreducible maps between modules for the mutation of $C$ stretch to give irreducible maps between modules for the mutation of $\tilde{C}$.
\end{lemma}

Let $T$ be the cluster tilting object for the cluster $C$. Let $\tilde{T}$ be the cluster tilting object for the cluster $\tilde{C}$.

Because $M'$ comes from $M$ via mutation, we get an exact sequence
$$M' \xrightarrow{\alpha} M'' \xrightarrow{\beta} M.$$
This exact sequence is characterized by the map $\beta$ coming from the irreducible maps in the category $\End(T)$ which are encoded in the quiver for the cluster $C$ \cite{JKS22}. This gives that this is an exact sequence of $\A_{k,n}$-modules which is moreover a non-split sequence where $\beta$ is a minimal right $\add$-$(T/M)$-approximation of $M$. This exact sequence then determines $M'$ uniquely \cite{JKS}.

Then because $C$ behaves well under stretching,
$$U_i(M') \xrightarrow{U_i(\alpha)} U_i(M'') \xrightarrow{U_i(\beta)} U_i(M)$$
is also an exact sequence where $U_i(\beta)$ consists of the appropriate irreducible maps dictated by the quiver for $\tilde{C}.$ Thus it is also a non-split sequence, and $U_i(\beta)$ is a minimal right $\add$-$(\tilde{T}/U_i(M))$-approximation of $U_i(M).$ This uniquely characterizes the module corresponding to the mutation of $F_i^*(f)$, which then must be $U_i(M')$ using the results of \cite{JKS,JKS22}, which ultimately depend on \cite{GLS}.

After mutation, let the cluster tilting objects be $T'$ and $\tilde{T}'$. Let us now show that the irreducible maps are also compatible with stretching.

The irreducible maps between the irreducible summands of $T'$ (respectively, $\tilde{T}'$) come from irreducible maps between irreducible summands of $T$ (respectively, $\tilde{T}$), compositions of such maps, or the map $\alpha$ (respectively, $U_i(\alpha)$). In all cases, we see that irreducible maps are preserved under the stretching operation.

Note that all this continues to hold if the roles of $M$ and $M'$ are switched, either by repeating the proof or using that these categories are $2$-Calabi-Yau. This completes the proof that behaving well under stretching is stable under mutation.

We now return to the proof of Lemma~\ref{stretchinglemma}. We have seen that if the compatibility between $U_i$ and $F_i^*$ holds for the functions in the cluster $C$, it holds for the functions in any mutation of $C$. Hence, by recursion, it holds for all cluster variables.
\end{proof}

\section{Rank 2 webs as cluster variables}\label{sec:rank2webs}

Here we show that the webs in $\mathcal{W}$ are in fact cluster variables. First, we consider some rank $2$ webs which give functions on $\Gr(k,2k)$.

\begin{definition}
Let $W(R,S,T,V)$ be a web with $|R|=a, |S|=b, |T|=c, |V|=0$ and $a+b+c=2k$. In other words, $R, S, T$ form a partition of $[2k]$, and the web $W(R,S,T,V)$ has exactly one leaf at each of the boundary vertices. We call such a web \emph{primitive}.
\end{definition} 

Any web in $\mathcal{W}$ comes from stretching a primitive web $W(R,S,T,V)$, see Section~\ref{sec:stretching}.

We previously saw that stretching a cluster variable on $\Gr(k,n)$ produces a cluster variable on $\Gr(k,n+1)$ or $\Gr(k+1,n+1)$. Thus we only need to show that the primitive webs are cluster variables.

In order to do this, we use that there is a natural map
$$\phi: \Gr(k,2k) \rightarrow \Conf_4 \A.$$

Here $\A = \SL_k/U$, where $U\subset SL_k$ is the maximal unipotent subgroup. $\A$ parameterizes complete flags in $\mathbb{C}^k$ with volume forms on each subspace. The map is defined as follows. Let $v_1, \dots, v_{2k}$ be the columns of the matrix representing a point of $\Gr(k,2k)$. We send this to the flags
$$F_1 := (v_1, v_2, \dots, v_{k-1}),$$
$$F_2 := (v_k, -v_{k-1}, v_{k-2} \dots, (-1)^{k-2}v_{2}),$$
$$F_3 := (v_{k+1}, v_{k+2}, \dots, v_{2k-1}),$$
$$F_4 := (v_{2k}, -v_{2k-1}, v_{2k-2} \dots, (-1)^{k-2}v_{k+2}),$$
where, for example, when we write $F_1 = (v_1, v_2, \dots, v_{k-1})$, we mean that the $i$-dimensional subspace of $F_1$ is spanned by $v_1, \dots, v_i$ and carries the volume form $v_1 \wedge \dots \wedge v_i$. Note that the vectors in $F_2$ and $F_4$ occur with alternative signs, so that the $i$-form for $F_2$ is $v_{k-i+1} \wedge \dots \wedge v_k$.

\begin{lemma} (see \cite{FL19}) The cluster structure on $\Gr(k,2k)$ can be pulled back from the cluster structure on $\Conf_4 \A$. In particular, if $f$ is a cluster variable on $\Conf_4 \A$, then $\phi^*(f)$ is a cluster variable on $\Gr(k,2k)$.
\end{lemma}

In fact, the cluster structures on these spaces are quasi-isomorphic (isomorphic up to the frozen variables). Any cluster variable on $\Conf_4 \A$ pulls back to a cluster variable on $\Gr(k,2k)$; the only wrinkle is that some frozen variables on $\Conf_4 \A$ become identified with a single frozen variable on $\Gr(k,2k)$.

The cluster structure on $\Conf_4 \A$ is quasi-isomorphic to the one on the double Bruhat cell $G^{w_0,w_0}$ (here $w_0$ is the longest word in the Weyl group $S_k$ of the Lie group $\SL_k$). The cluster structure on $G^{w_0,w_0}$ was studied extensively by Berenstein, Fomin and Zelevinsky \cite{BFZ} and this quasi-isomorphism is folklore; it is implicit in Fock-Goncharov~\cite{FG06}, see also~\cite{LL23}.

In \cite{BFZ}, they show that cluster variables can be written down using the combinatorics of double reduced words. We consider the Weyl group $S_k \times S_k$, where we generate the first factor by simple reflections $1, 2, \dots, k-1$ and we generate the second factor by simple reflections $\overline{1}, \overline{2}, \dots, \overline{k-1}$. One chooses a double reduced word for $(w_0,w_0)$, which is a word in the alphabet $1, 2, \dots, k-1, \overline{1}, \overline{2}, \dots, \overline{k-1}$. Corresponding to any such word is a cluster:

\begin{theorem}~\cite{BFZ} Fix any reduced word for $(w_0,w_0) \in S_k \times S_k$. Consider any partial word formed by the first $N$ letters of the double reduced word. The product of this partial word gives $(u,v) \in S_k \times S_k.$ Then there exists a cluster for $G^{w_0,w_0}$ consisting of generalized minors $\Delta_{u\omega_i,v\omega_i}$ where $(u,v)$ runs over such products and $\omega_i$ runs over fundamental weights.
\end{theorem}

In \cite{LL23}, Luo and the first author analyze the corresponding clusters on $\Conf_4 SL_k$. Any cluster variable on $\Conf_4 \A$ has an associated quadruple of weights $(\kappa, \lambda, \mu, \nu)$, which are determined by how the function transforms under the natural actions of the diagonal torus $T$ on the flags $F_1, F_2, F_3, F_4$. For example, a function of weight $\omega_i$ in the flag $F_j$ depends on the $i$-form on the $i$-dimensional space of $F_j$.

\begin{proposition}~\cite{LL23} There is a natural map $\tilde{\phi}: G^{w_0,w_0} \rightarrow \Conf_4 \A$ such that if $f$ is a cluster variable on  $\Conf_4 \A$, then $\tilde{\phi}^*(f)$ is a cluster variable on $G^{w_0,w_0}$. There are cluster variables $f_{u\omega_i,v\omega_i}$ such that $\Delta_{u\omega_i,v\omega_i} = \tilde{\phi}^*(f_{u\omega_i,v\omega_i})$.
\end{proposition}

\begin{theorem}~\cite{LL23} The function $f_{u\omega_i,v\omega_i}$ has weights $(\kappa, \lambda, \mu, \nu)$ where $\kappa, \lambda, \mu, \nu$ are the minimal dominant weights such that 
\[u\omega_i = \kappa + w_0\nu\]
\[v\omega_{k-i} = \lambda + w_0\mu\].
\end{theorem}

Note that these formulas are slightly different from those at the end of Section 6 of \cite{LL23}; this is because our conventions and normalizations are slightly different.

\begin{theorem}~\label{thm:web-cluster}
Any primitive web in $\mathcal{W}$ is a cluster variable.
\end{theorem}

\begin{proof}

Consider any web $W(R,S,T,V)$ with $|R|=a, |S|=b, |T|=c, |V|=0$ and $a+b+c=2k$, $a, b, c < k$.

Let $i=a+c-k$. Note that $b=(k-a) + (k-c)$. Now, let us choose a double reduced word which contains as a partial word some permutation $(u,v)$ such that 
$$u\omega_i = 
(\overbrace{0, \cdots, 0}^{k-c}, \overbrace{1, \cdots, 1}^{a+c-k} , \overbrace{0, \cdots, 0}^{k-a} ) = 
\omega_a + w_0\omega_c,$$

$$v\omega_{k-1} = (  \overbrace{1, \cdots, 1}^{k-c} , \overbrace{0, \cdots, 0}^{a+c-k} , \overbrace{1, \cdots, 1}^{k-c}
) = \omega_{k-a} + w_0 \omega_{k-c}.$$
There are many choices of reduced word that give rise to such $u$ and $v$.

Then the function $f_{u\omega_i,v\omega_i}$ depends on $a$-form $v_1 \wedge \dots \wedge v_a$ from $F_1$, the $(k-a)$-form $v_{a+1} \wedge \dots \wedge v_k$ from $F_2$, the $(k-c)$-form $v_{k+1} \wedge \dots \wedge v_{2k-c}$ from $F_3$, and the $c$-form $v_{2k-c+1} \wedge \dots \wedge v_{2k}$ from $F_4$:

\[
\xymatrix{ v_1 \wedge \dots \wedge v_a \in F_1 \ar@{-}[r] & v_{2k-c+1} \wedge \dots \wedge v_{2k} \in F_4 \ar@{-}[d] \\
           v_{a+1} \wedge \dots \wedge v_k \in F_2 \ar@{-}[u] & v_{k+1} \wedge \dots \wedge v_{2k-c} \in F_3 \ar@{-}[l] }
\]

Invariant theory tells us that there is a unique invariant $t$ of these forms (up to scaling), which is given by the web in Figure~\ref{fig:flag_web}.

This web is a cluster variable on $\Conf_4 \A$ which pulls back to a cluster variable on $\Gr(k,2k)$ given by the web:

\begin{figure}[H]
    \centering
    \includegraphics[scale=0.9]{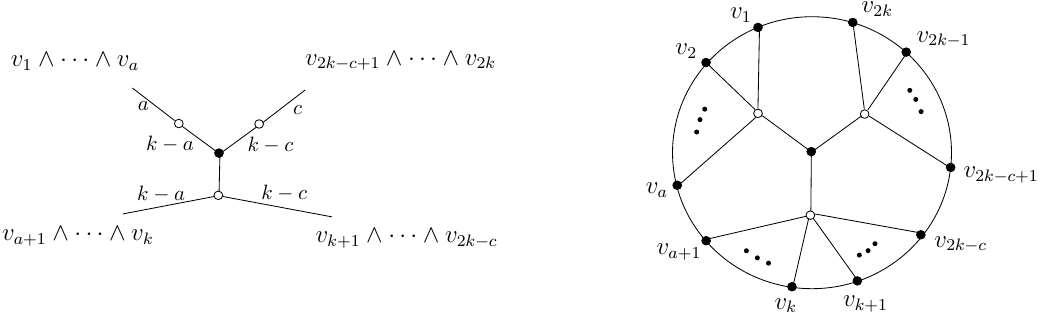}
    \caption{The corresponding web to the invariant $t$.}
    \label{fig:flag_web}
\end{figure}

Up to rotation (which is a symmetry of the cluster structure), and by varying $a, b, c$, we have exhibited all the primitive webs on $\Gr(k,2k)$ as cluster variables.

\end{proof}

\begin{corollary}~\label{cor:allGrassm}
    Any web in $W \in \mathcal{W}$ is a cluster variable.
\end{corollary}
   
Therefore, the modules $\Psi(W)$ for $W \in \mathcal{W}$ are all reachable. It is an open question whether all rigid indecomposable modules are reachable via mutation. Thus Corollary~\ref{cor:allGrassm} suggests the possibility that all rigid indecomposable modules are reachable by mutation.

\section{Cactus sequence and Main Result}\label{sec:cactusmain}

The cactus sequence is a specific sequence of mutations in the space of configurations of three flags $\Conf_3 \mathcal{A}$ which was studied in \cite[Subsection 3.3]{L19}. It is a maximal green sequence on $\Conf_3 \mathcal{A}$, and its tropicalization induces the Sch\"utzenberger involution. Thus it is a very natural mutation sequence to study.

Recall that $\A$ is the principal flag variety $G/U$ for $G=SL_k$. It turns out the one can find a copy of this cluster structure ``inside'' $\Gr (k,3(k-1))$ (see the Lemma~\ref{clusterinside}). Thus we can perform the cactus sequence within the cluster structure on $\Gr (k,3(k-1))$. This sequence will allow us to access many of the rank $2$ functions on the Grassmannian. In fact, up to stretching, these functions will give us all the rank $2$ webs in $\mathcal{W}$.

We use that there is a natural map
$$\phi: \Gr(k,3(k-1)) \rightarrow \Conf_3 \A.$$

The map is defined as follows. Let $v_1, \dots, v_{3k-3}$ be the columns of the matrix representing a point of $\Gr(k,3k-3)$. We send this to the flags 
$$F_1 := (v_1, v_2, \dots, v_{k-1}),$$
$$F_2 := (v_k, v_{k+1}, \dots, v_{2k-2}),$$
$$F_3 := (v_{2k-1}, v_{2k}, \dots, v_{3k-3}).$$
in the notation of Section~\ref{sec:rank2webs}.

\begin{lemma}\label{clusterinside} (see \cite{FL19}) Let $C$ be a cluster on $\Conf_3 \A$. Then there exists a cluster $C'$ on $\Gr(k,3k-3)$ such that $\{\phi^*(f) | f \in C\} \subset C'$. Moreover the quiver for $C'$ restricts to the quiver for $C$. Thus we can view the cluster structure on $\Conf_3 \A$ is a cluster substructure of the cluster structure on $\Gr(k,3k-3)$.
\end{lemma}

The proof of Lemma~\ref{stretchinglemma} is inspired by this result, in fact we use the same idea.

Now let us define the functions appearing in the cactus sequence. Let $a, b, c, d \leq k$, $a>d$, $a+b+c+d=2k$. Moreover, assume that $a-d, b, c$ satisfy the triangle inequality. For any such $a,b,c,d$ we define the function $\tcfr{a,d}{b}{c}$ on $\Conf_3 \A$. It is determined by its pullback to the Grassmannian $\Gr (k,3(k-1))$, $\phi^*(\tcfr{a,d}{b}{c})$. Because we will be primarily working with functions on the Grassmannian, let us simply call this pullback function $\tcfr{a,d}{b}{c}$ as well.

Let
$$R=\{d+1, d+2, \dots, a\},$$
$$S=\{k, k+1, \dots, k+b-1\},$$
$$T=\{2k-1, 2k, \dots, 2k+c-2\},$$
$$V=\{1, 2, \dots, d\}$$

\begin{definition} Define
\begin{enumerate}
    \item $\tcfr{a,d}{b}{c} := W(R,S,T,V)$.
    \item The function $\tcfr{a}{b}{c}$ for $a+b+c=k$ and $a, b, c \leq k-1$ is the Pl\"ucker coordinate involving $1, 2, \dots, a, k, k+1, \dots k+b-1, 2k-1, 2k, \dots, 2k+c-2$.
\end{enumerate}
It is clear from the definition that the function $W(R,S,T,V)$ descends to a function on $\Conf_3 \A$. 
\end{definition}

\begin{figure}[H]
    \centering
    \includegraphics[width=12cm]{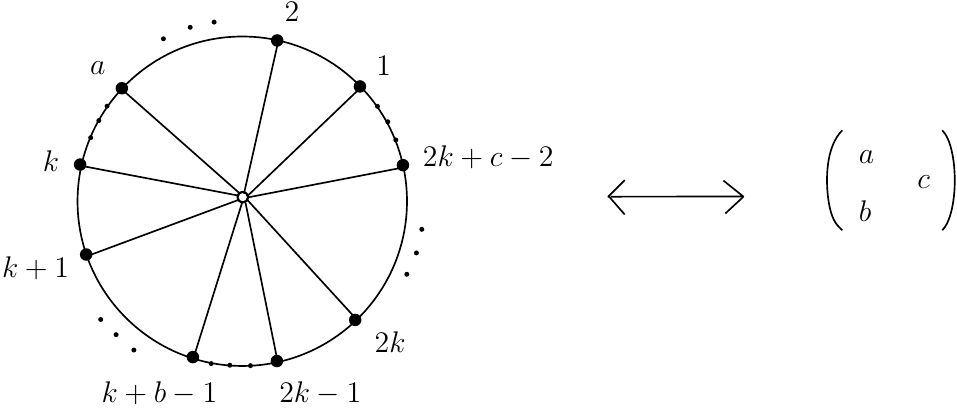}
    \caption{Rank $1$ webs.}
    \label{webPlucker}
\end{figure}

\begin{example}
    
The following webs correspond to the functions $\bigl(\begin{smallmatrix}
4,2& \\ &2\\ 2& 
\end{smallmatrix} \bigr)$, $\bigl(\begin{smallmatrix}
4& \\ &4\\ 2& 
\end{smallmatrix} \bigr)$, $\bigl(\begin{smallmatrix}
3,1& \\ &3\\ 3& 
\end{smallmatrix} \bigr)$ in $\Gr(5,12)$. Notice that, up to the leaves intersecting, these webs are trees, so the fact that they are cluster variables in $\Gr(5,12)$ confirms the philosophy of Fomin and Pylyavskyy \cite[Conjecture 10.1]{FP}. We can see that these functions, via stretching, can live in many Grassmannians, and they will be cluster variables thanks to Theorem~\ref{thm:web-cluster}.

\begin{figure}[H]
    \centering
    \includegraphics[width=14cm]{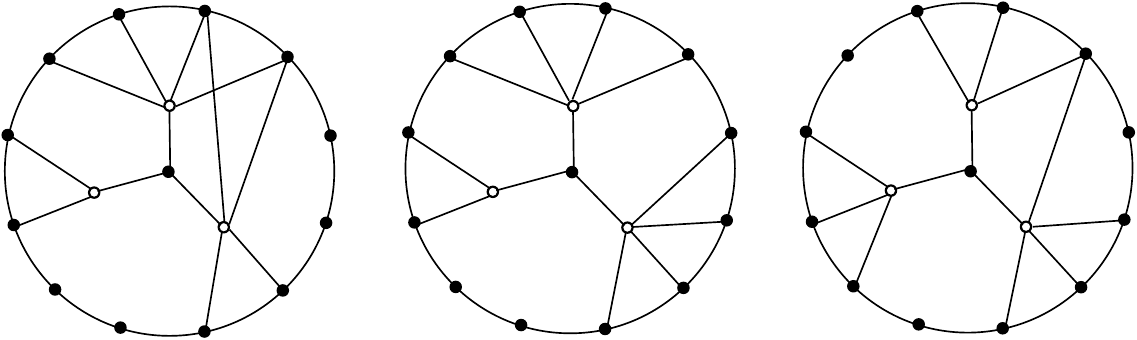}
    \caption{Rank $2$ webs.}
    \label{web1}
\end{figure}

\end{example}

The cactus sequence begins with a cluster that includes all of the functions $\tcfr{a}{b}{c}$ for $a+b+c=k$. Over the course of the mutations in the cactus sequence, one sees all of the functions of the form $\tcfr{a,d}{b}{c}$ defined above. At the end of the sequence, one is left with the functions $\tcfr{a,d}{b}{c}$ where $d=0$, and these are the all the rank $2$ webs in $\mathcal{W}$, up to stretching.

\begin{center}
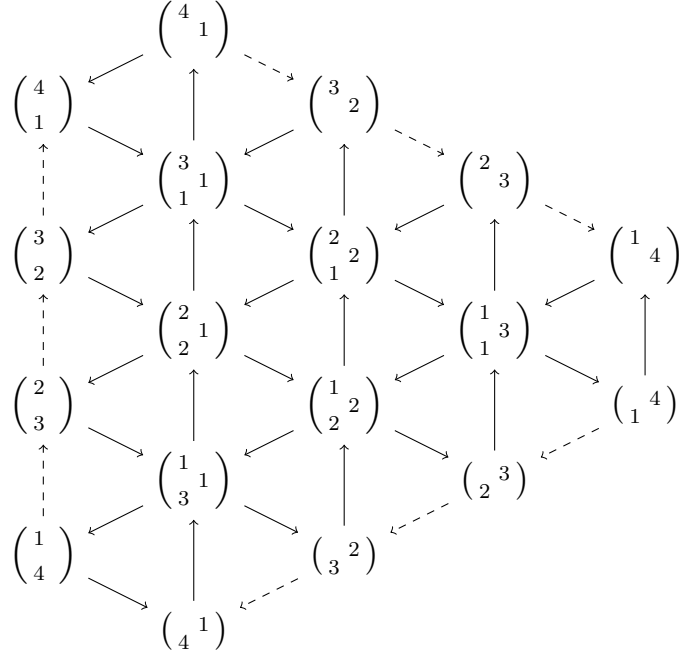
\begin{figure}[H]
\begin{tikzpicture}[scale=2]
  \node (x410) at (-2,0) {$\tcfr{4}{1}{}$};
  \node (x320) at (-2,-1) {$\tcfr{3}{2}{}$};
  \node (x230) at (-2,-2) {$\tcfr{2}{3}{}$};
  \node (x140) at (-2,-3) {$\tcfr{1}{4}{}$};

  \node (x401) at (-1,0.5) {$\tcfr{4}{}{1}$};
  \node (x311) at (-1,-0.5) {$\tcfr{3}{1}{1}$};
  \node (x221) at (-1,-1.5) {$\tcfr{2}{2}{1}$};
  \node (x131) at (-1,-2.5) {$\tcfr{1}{3}{1}$};
  \node (x041) at (-1,-3.5) {$\tcfr{}{4}{1}$};

  \node (x302) at (0,0) {$\tcfr{3}{}{2}$};
  \node (x212) at (0,-1) {$\tcfr{2}{1}{2}$};
  \node (x122) at (0,-2) {$\tcfr{1}{2}{2}$};
  \node (x032) at (0,-3) {$\tcfr{}{3}{2}$};

  \node (x203) at (1,-0.5) {$\tcfr{2}{}{3}$};
  \node (x113) at (1,-1.5) {$\tcfr{1}{1}{3}$};
  \node (x023) at (1,-2.5) {$\tcfr{}{2}{3}$};

  \node (x104) at (2,-1) {$\tcfr{1}{}{4}$};
  \node (x014) at (2,-2) {$\tcfr{}{1}{4}$};

  \draw [->] (x014) to (x104);
  \draw [->] (x023) to (x113);
  \draw [->] (x113) to (x203);
  \draw [->] (x032) to (x122);
  \draw [->] (x122) to (x212);
  \draw [->] (x212) to (x302);
  \draw [->] (x041) to (x131);
  \draw [->] (x131) to (x221);
  \draw [->] (x221) to (x311);
  \draw [->] (x311) to (x401);
  \draw [->, dashed] (x140) to (x230);
  \draw [->, dashed] (x230) to (x320);
  \draw [->, dashed] (x320) to (x410);

  \draw [->] (x401) to (x410);
  \draw [->] (x302) to (x311);
  \draw [->] (x311) to (x320);
  \draw [->] (x203) to (x212);
  \draw [->] (x212) to (x221);
  \draw [->] (x221) to (x230);
  \draw [->] (x104) to (x113);
  \draw [->] (x113) to (x122);
  \draw [->] (x122) to (x131);
  \draw [->] (x131) to (x140);
  \draw [->, dashed] (x014) to (x023);
  \draw [->, dashed] (x023) to (x032);
  \draw [->, dashed] (x032) to (x041);

  \draw [->] (x140) to (x041);
  \draw [->] (x230) to (x131);
  \draw [->] (x131) to (x032);
  \draw [->] (x320) to (x221);
  \draw [->] (x221) to (x122);
  \draw [->] (x122) to (x023);
  \draw [->] (x410) to (x311);
  \draw [->] (x311) to (x212);
  \draw [->] (x212) to (x113);
  \draw [->] (x113) to (x014);
  \draw [->, dashed] (x401) to (x302);
  \draw [->, dashed] (x302) to (x203);
  \draw [->, dashed] (x203) to (x104);

\end{tikzpicture}
\caption{The initial quiver for $\Conf_3 \A$ for $SL_5$.}
\end{figure}
\end{center}

The particulars of this mutation sequence, like the precise sequence of mutations and the clusters that occur along the way, are not so important for us. What is important for us is that each mutation has the following form:

\newpage

\begin{proposition}\label{cactusmutation}~\cite[Proposition 3.2]{L19} 

Every mutation in the cactus sequence transforms cluster variables
\[\tcfr{a,d}{b}{c} \rightsquigarrow \tcfr{a-1,d-1}{b+1}{c+1}\]
via the identity

\[\tcfr{a,d}{b}{c} \tcfr{a-1,d-1}{b+1}{c+1}=\tcfr{a-1,d}{b+1}{c} \tcfr{a,d-1}{b}{c+1} + \tcfr{a-1,d}{b}{c+1} \tcfr{a,d-1}{b+1}{c}\]
with $a,b,c,d$ satisfying the appropriate inequalities for all the terms to make sense.

\end{proposition}

Some of the mutation identities will be degenerations of the above identity using one of the following identities:
\[\tcfr{a,d}{b}{c} = \tcfr{a}{b}{}\tcfr{d}{}{c} \textrm{ if } a+b=c+d=k. \]
\[\tcfr{a,d}{b}{c} = \tcfr{a}{}{c}\tcfr{d}{c}{} \textrm{ if } a+c=b+d=k.\]
\[\tcfr{k,d}{b}{c} = \tcfr{d}{b}{c} \textrm{ if } b+c+d=k.\]
Note that in Proposition~\ref{cactusmutation}, all the mutations involve rank $2$ functions. However, we can use the degeneracy to view rank $1$ functions (Pl\"ucker coordinates) as rank $2$ functions: Thus we start with the function $\tcfr{d}{b}{c} = \tcfr{k,d}{b}{c}$ and mutate until we end up with the function $\tcfr{k-d,0}{b+d}{c+d}$.

For example, there is a mutation
$$\tcfr{k-2}{1}{1}\tcfr{k-1,k-3}{2}{2}=\tcfr{k,k-2}{1}{1}\tcfr{k-1,k-3}{2}{2}=
\tcfr{k-1,k-2}{2}{1}\tcfr{k,k-3}{1}{2}+\tcfr{k-1,k-2}{1}{2}\tcfr{k,k-3}{2}{1}=$$
$$\tcfr{k-1}{}{1}\tcfr{k-2}{2}{}\tcfr{k-3}{1}{2}+\tcfr{k-1}{1}{}\tcfr{k-2}{}{2}\tcfr{k-3}{2}{1}.$$
The first and last identities are degeneracies, while the middle identity is the cactus mutation.

In the cases where 
\[\tcfr{a,d}{b}{c} = \tcfr{a}{b}{}\tcfr{d}{}{c},\]
the module one would attach to the cluster variable $\tcfr{a,d}{b}{c}$ (using a degenerate case of the procedure we described for $\Psi$) is the direct sum of the modules attached to $\tcfr{a}{b}{}$ and $\tcfr{d}{}{c}$. The analogous statement is true when 
\[\tcfr{a,d}{b}{c} = \tcfr{a}{}{c}\tcfr{d}{c}{} \textrm{ if } a+c=b+d=k.\]
Thus these cases can be treated uniformly in the proof of Theorem~\ref{thm:ses} below. The final degeneracy changes the profiles of the modules in proof of Theorem~\ref{thm:ses}, but it actually simplifies the arguments.

Because the functions in this sequence are all of the ones in $\mathcal{W}$, up to stretching, by Lemma~\ref{stretchinglemma} it is enough to understand the modules attached to the functions in this sequence. Our strategy is to use the cactus mutation sequence to recursively verify that the module corresponding to $W \in \mathcal{W}$ is $\Psi(W)$ for all the functions $\tcfr{a,d}{b}{c}$.

At any stage in the cactus sequence, we know the modules that categorify the cluster variables 
\[\tcfr{a,d}{b}{c}, \tcfr{a-1,d}{b+1}{c}, \tcfr{a,d-1}{b}{c+1}, \tcfr{a-1,d}{b}{c+1}, \tcfr{a,d-1}{b+1}{c}\]
are the ones coming from the map $\Psi$. From this, we would like to compute the module which categorifies $\tcfr{a-1,d-1}{b+1}{c+1}$. 

In order to do this, we will first show that $\Psi\tcfr{a-1,d-1}{b+1}{c+1}$ sits in the appropriate exact sequences:

\begin{proposition}~\label{thm:ses} The following are short exact sequences of modules in the categorification of $\Gr(k,3(k-1))$:

\[0\rightarrow\Psi\tcfr{a,d}{b}{c} \rightarrow \Psi\tcfr{a-1,d}{b}{c+1}\oplus \Psi\tcfr{a,d-1}{b+1}{c}\rightarrow \Psi\tcfr{a-1,d-1}{b+1}{c+1}\rightarrow 0 \]

\[0\rightarrow\Psi\tcfr{a-1,d-1}{b+1}{c+1}\rightarrow \Psi\tcfr{a-1,d}{b+1}{c}\oplus \Psi\tcfr{a,d-1}{b}{c+1}\rightarrow \Psi\tcfr{a,d}{b}{c}\rightarrow 0 \]

\end{proposition}

\begin{proof} Recall that we may write the profile of the rank $2$ module $\Psi\tcfr{a,d}{b}{c}$ as a stack of contours:

$$\frac{D^{d+x}U^{y+k-a}D^{y}U^{z+k-b}D^{z}U^{x+k-c}}{D^{d}U^{x}D^{y}U^{k-a+y}D^{z}U^{k-b+z}D^{x}U^{k-c}} = $$
$$\left(\frac{D}{D}\right)^{d} \left(\frac{D}{U}\right)^{x} \left(\frac{U}{D}\right)^{y} \left(\frac{U}{U}\right)^{k-a} \left(\frac{D}{U}\right)^{y} \left(\frac{U}{D}\right)^{z} \left(\frac{U}{U}\right)^{k-b} \left(\frac{D}{U}\right)^{z} \left(\frac{U}{D}\right)^{x} \left(\frac{U}{U}\right)^{k-c} $$

Thus, for example, the first exact sequence in the theorem can be written as:
$$0 \rightarrow \frac{D^{d+x}U^{y+k-a}D^{y}U^{z+k-b}D^{z}U^{x+k-c}}{D^{d}U^{x}D^{y}U^{k-a+y}D^{z}U^{k-b+z}D^{x}U^{k-c}} \rightarrow$$
$$\frac{D^{d+x}U^{y+k-a}D^{y-1}U^{z+1+k-b}D^{z+1}U^{x+k-c-1}}{D^{d}U^{x}D^{y-1}U^{k-a+y}D^{z+1}U^{k-b+z+1}D^{x}U^{k-c-1}} \oplus 
\frac{D^{d-1+x}U^{y+1+k-a}D^{y+1}U^{z+k-b-1}D^{z}U^{x+k-c}}{D^{d-1}U^{x}D^{y+1}U^{k-a+y+1}D^{z}U^{k-b-1+z}D^{x}U^{k-c}}$$
$$\rightarrow \frac{D^{d-1+x}U^{y+k-a+1}D^{y}U^{z+k-b}D^{z+1}U^{x+k-c-1}}{D^{d-1}U^{x}D^{y}U^{k-a+1+y}D^{z+1}U^{k-b+z}D^{x}U^{k-c-1}} \rightarrow 0$$

Written this way, the exact sequence does not look very illuminating. Let us examine the exact sequence in a representative example, which will indicate how the general argument works. Take $a=5, b=4, c=4, d=1$. In this case we get that $x,y,z=2$. We are interested in the exact sequences

\[0\rightarrow\Psi\tcfr{5,1}{4}{4} \xrightarrow{f} \Psi\tcfr{4,1}{4}{5}\oplus \Psi\tcfr{5}{5}{4}\xrightarrow{g} \Psi\tcfr{4}{5}{5}\rightarrow 0 \]

\[0\rightarrow\Psi\tcfr{4}{5}{5}\xrightarrow{f'} \Psi\tcfr{4,1}{5}{4}\oplus \Psi\tcfr{5}{4}{5}\xrightarrow{g'} \Psi\tcfr{5,1}{4}{4}\rightarrow 0 \]

We wish to see these as exact sequences in $\Gr(7,18)$. However, none of the functions involved uses the vectors $v_6, v_{12}$ or $v_{18}$. Thus we may actually work in $\Gr(7,15)$ by eliminating $v_6, v_{12}, v_{18}$ and reindexing in order to make things slightly simpler. We will exhibit exact sequences of modules for the categorification of $\Gr(7,15)$ which can be stretched to give exact sequences of modules for the categorification of $\Gr(7,18)$.

We start with the first exact sequence. Below, we picture the modules for $\Psi\tcfr{4,1}{4}{5}$ and $\Psi\tcfr{5}{5}{4}$.

\begin{figure}[h]
    \centering
    \includegraphics[width=15cm]{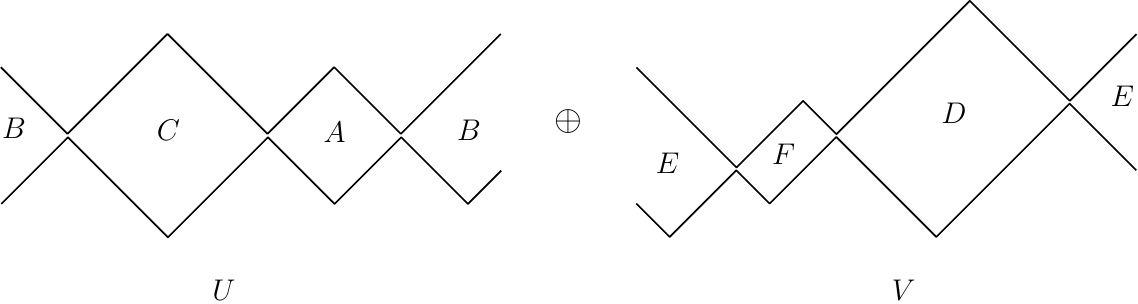}
\end{figure}

The regions of the diagram are labelled by vector spaces satisfying the following:
\begin{itemize}
    \item $U, V, W$ are two dimensional spaces;
    \item $A, B, C \subset U$ are one-dimensional spaces spanned by $a, b, c$, respectively, such that $a+b+c=0$.
    \item $D, E, F \subset V$ are one-dimensional spaces spanned by $d, e, f$, respectively, such that $d+e+f=0$.
    \item $G, H, I \subset W$ are one-dimensional spaces spanned by $g, h, i$, respectively, such that $g+h+i=0$.
\end{itemize}

We can now describe the map $f$ and show that it is injective. We map $\Psi\tcfr{5,1}{4}{4}$ to the direct sum of $\Psi\tcfr{4,1}{4}{5}$ and $\Psi\tcfr{5}{5}{4}$ by mapping

\begin{enumerate}[(i)]
    \item $g \rightarrow a+d$ so that $G\hookrightarrow A\oplus D \hookrightarrow D\oplus U$;
    \item $h \rightarrow b+e$ so that $H\hookrightarrow B\oplus E \hookrightarrow B\oplus V$;
    \item $i \rightarrow c+f$ so that $I\hookrightarrow C\oplus F \hookrightarrow C \oplus V$;
    \item the previous maps are compatible and determine a map $U\hookrightarrow V\oplus W$.
\end{enumerate}

In the following diagram, we see $\Psi\tcfr{5,1}{4}{4} \hookrightarrow \Psi\tcfr{4,1}{4}{5} \oplus \Psi\tcfr{5}{5}{4}$:

\begin{figure}[h]
    \centering
    \includegraphics[width=15cm]{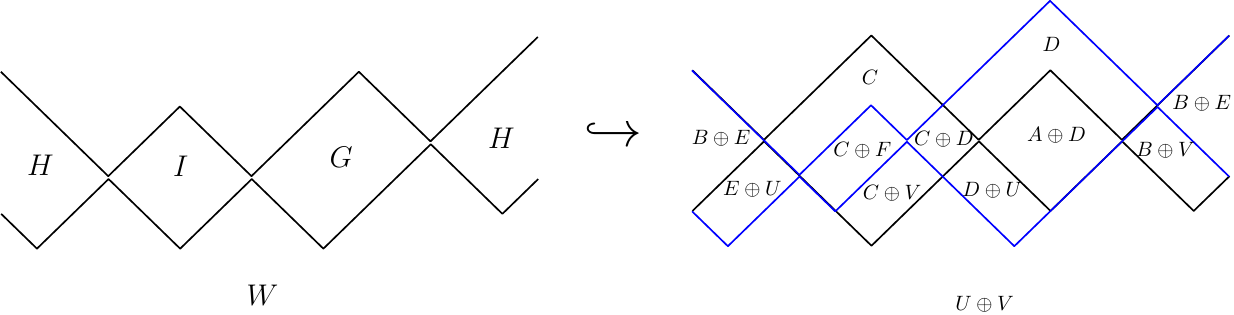}
\end{figure}

Now, we take the cokernel of this map $f$. We use the following natural isomorphisms that follow from linear algebra:

\begin{enumerate}[(i)]
    \item $A \sim \frac{A \oplus D}{K} \sim D$
    \item $B \sim \frac{B \oplus E}{H} \sim E$
    \item $C \sim \frac{C \oplus F}{K} \sim F$
    \item $\frac{U \oplus V}{W} \sim \frac{U \oplus D}{G}  \sim \frac{A \oplus V}{G} \sim \frac{U \oplus E}{H} \sim \frac{B \oplus V}{H} \sim \frac{U \oplus F}{I}  \sim \frac{C \oplus V}{I}$
\end{enumerate}
(Note that this list of identities contains more than we need in this particular case, but it is the complete list of identities one uses in the general case.)

From this, we obtain the following module which is the desired result:
\begin{figure}[h]
    \centering
    \includegraphics[width=10cm]{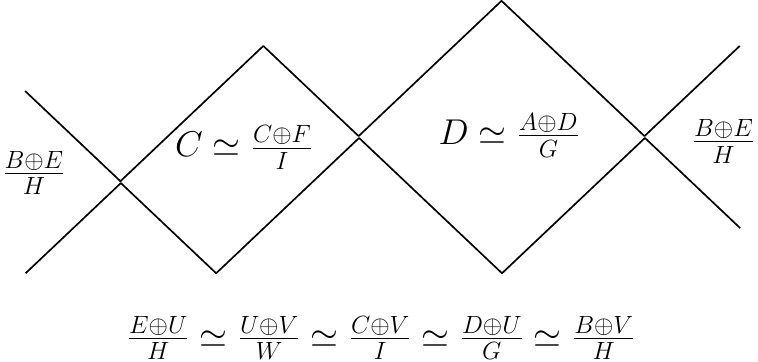}
\end{figure}
The module is the rank $2$ indecomposable module that we seek, corresponding to $\Psi(\tcfr{4}{5}{5}).$

We can similarly check the second exact sequence. Below are the modules $\Psi\tcfr{4,1}{5}{4}$, $\Psi\tcfr{5}{4}{5}$ and their direct sum $\Psi\tcfr{4,1}{5}{4}\oplus \Psi\tcfr{5}{4}{5}$:

\begin{figure}[H]
    \centering
    \includegraphics[width=15cm]{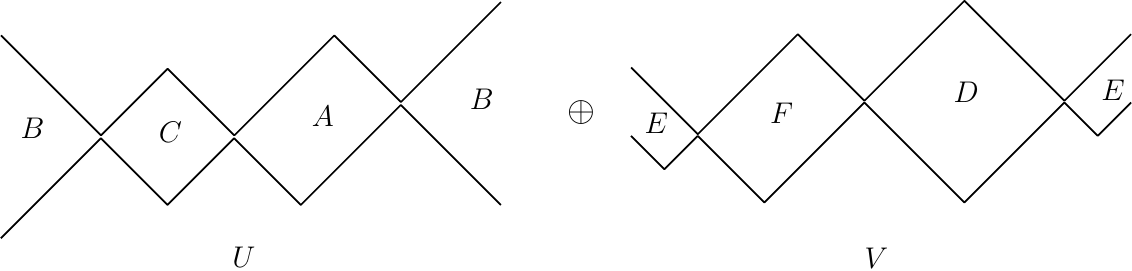}
\end{figure}

The module $\Psi\tcfr{4}{5}{5}$ includes into this direct sum via the map $f'$, defined as in the previous case:

\begin{figure}[H]
    \centering
    \includegraphics[width=15cm]{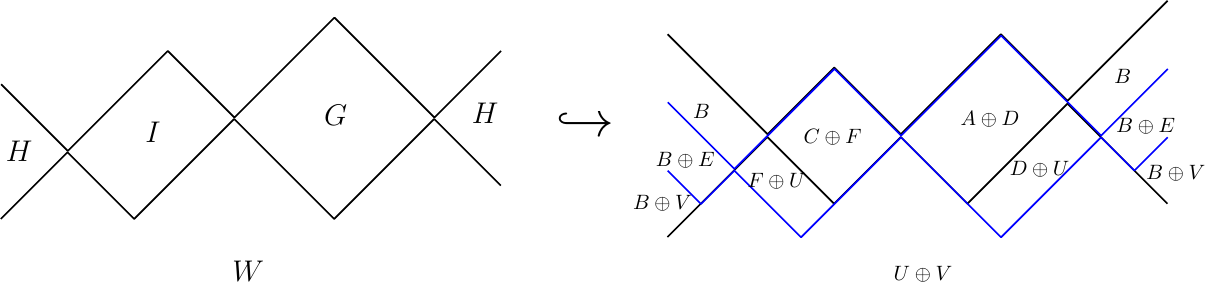}
\end{figure}

Using the same identities, we see that the quotient is the module $\Psi\tcfr{5,1}{4}{4}$:

\begin{figure}[H]
    \centering
    \includegraphics[width=10cm]{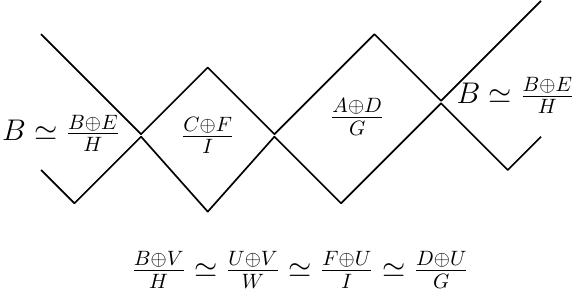}
\end{figure}

Thus once we know that the functions $\tcfr{5,1}{4}{4}, \tcfr{4,1}{5}{4}, \tcfr{5}{4}{5},\tcfr{4,1}{4}{5}, \tcfr{5}{5}{4}$ all correspond to the modules given by applying $\Psi$, we see that this is also the case for $\tcfr{4}{5}{5}$.

The general argument will use the exact same diagrams, with the proportions of various rectangles scaled appropriately. One needs to make slight modifications in the case of the degenerate exact sequences, but these sequences are actually simpler.

\end{proof}

\begin{proposition}
    For the clusters appearing in the cactus sequence, the irreducible maps come in four types (up to degeneration):
    \begin{itemize}
        \item[(i)] $\Psi\tcfr{a,d}{b}{c} \rightarrow \Psi\tcfr{a-1,d}{b}{c+1}$
        \item[(ii)] $\Psi\tcfr{a,d}{b}{c} \rightarrow \Psi\tcfr{a,d-1}{b+1}{c}$
        \item[(iii)] $\Psi\tcfr{a,d}{b}{c} \rightarrow \Psi\tcfr{a,d+1}{b}{c-1}$
        \item[(iv)] $\Psi\tcfr{a,d}{b}{c} \rightarrow \Psi\tcfr{a+1,d}{b-1}{c}$
    \end{itemize}
    These maps are given by the maps in Figure~\ref{fig:irr-maps} up to stretching.
\end{proposition}

\begin{figure}[h]
    \centering
    \includegraphics[width=0.5\linewidth]{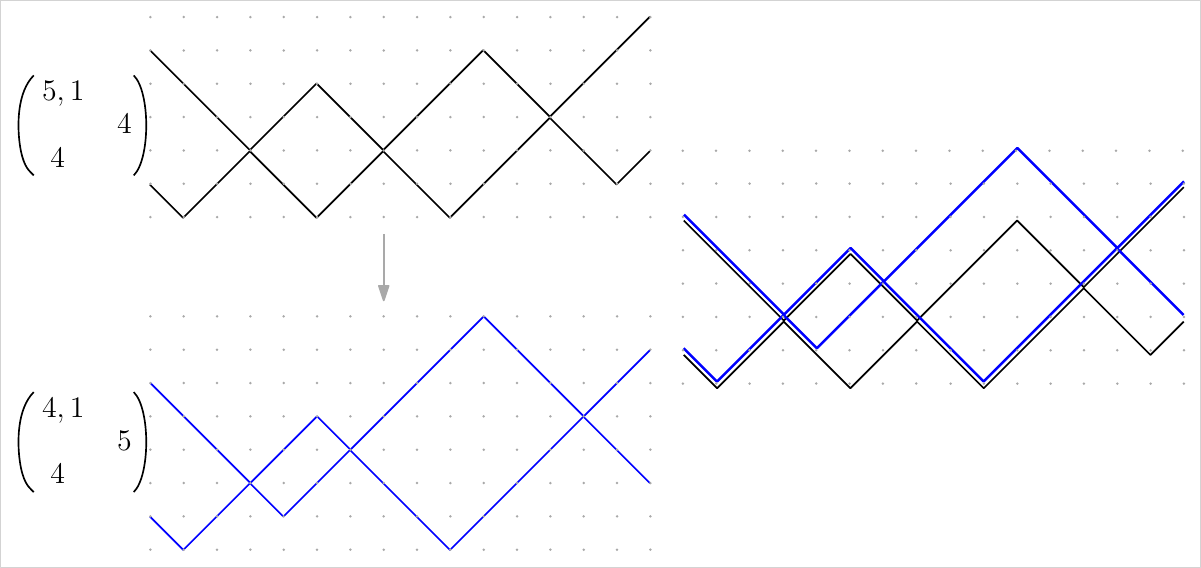}\includegraphics[width=0.5\linewidth]{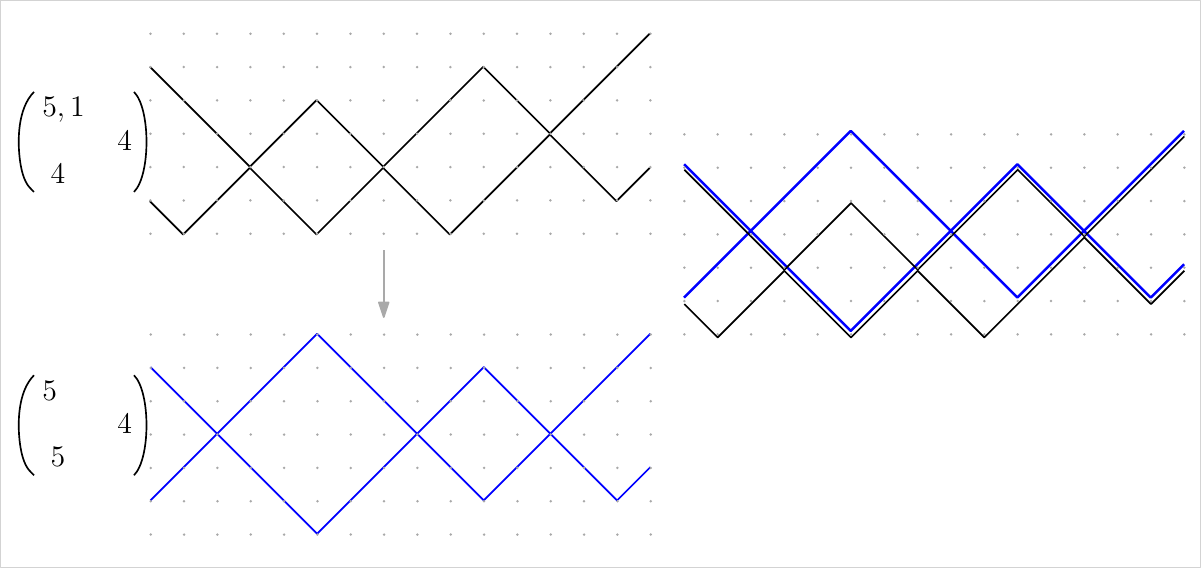}
    \includegraphics[width=0.5\linewidth]{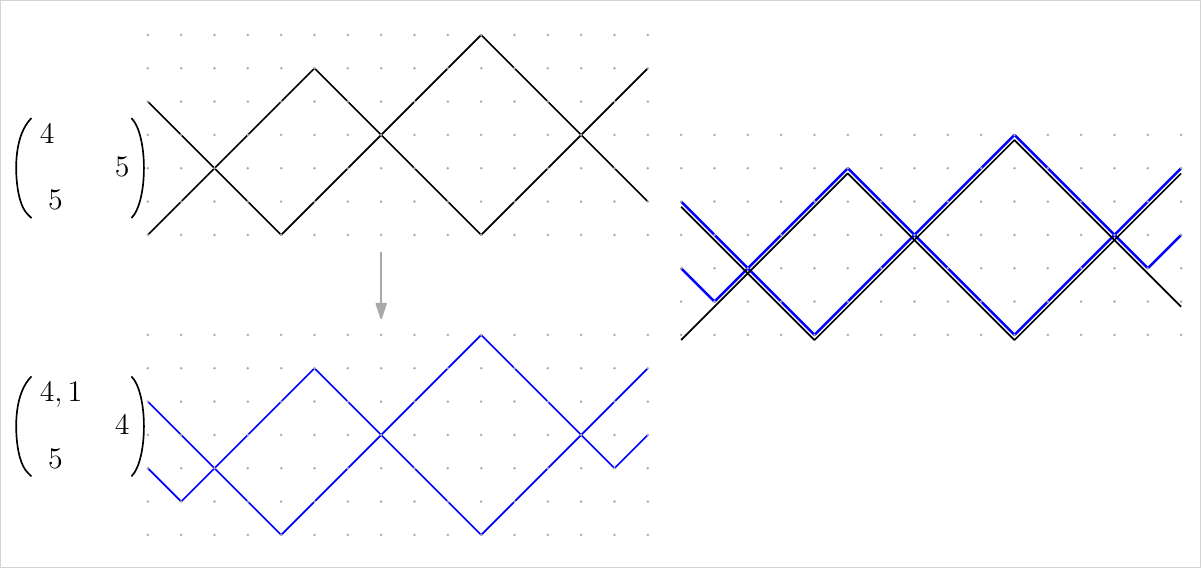}\includegraphics[width=0.5\linewidth]{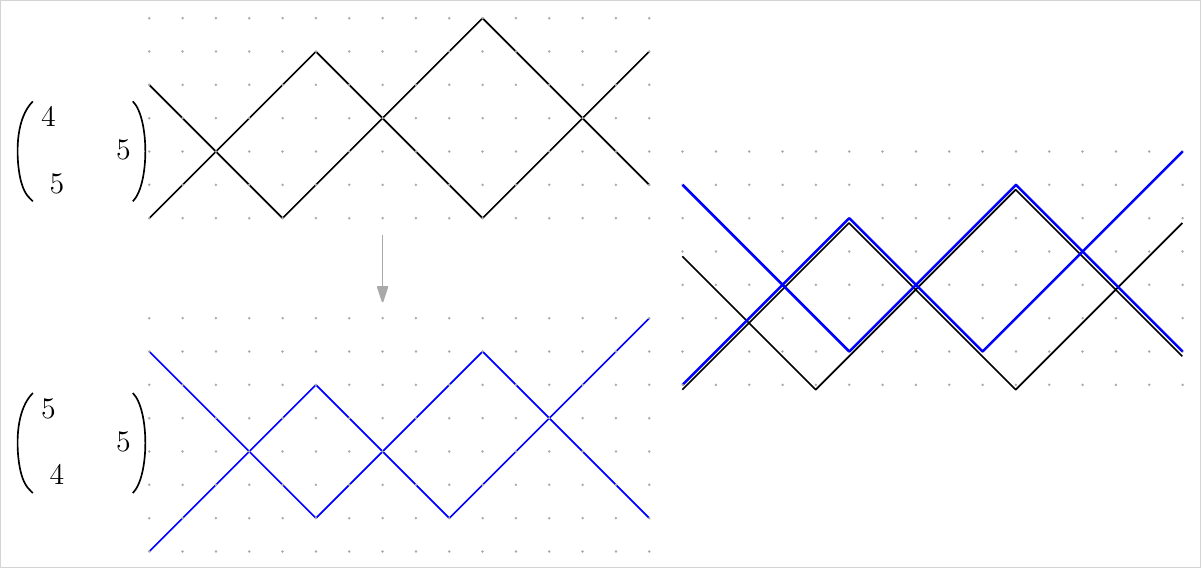}
    \caption{Four different irreducible maps in the cactus sequence.}
    \label{fig:irr-maps}
\end{figure}

\begin{proof}

These maps can be tracked through the cactus sequence inductively. As explained in~\cite[Proposition 3.2]{L19}, there are four types of mutations in the sequence including degenerations. For each of these types of mutations, one should check if the irreducible maps are as stated in the proposition before and after the mutation. This is a finite check.

\end{proof}

\begin{corollary}
    The mutation sequences in Proposition~\ref{thm:ses} are non-split exact sequences giving minimal left (respectively right) $\add$-$(T/M)$-approximations of $M=\Psi\tcfr{a,d}{b}{c}$ for the appropriate cluster-tilting object $T$.
\end{corollary}

In other words, they allow us to inductively conclude that if $\Psi\tcfr{a,d}{b}{c}$ categorifies $\tcfr{a,d}{b}{c}$, then $\Psi\tcfr{a-1,d-1}{b+1}{c+1}$ categorifies $\tcfr{a-1,d-1}{b+1}{c+1}$. From this we may conclude:

\begin{theorem}
       The pull-back of the function $\tcfr{a,d}{b}{c}$ from $\Conf_3 \A$ to $\Gr(k,3k-3)$ is a cluster variable whose categorification is given by the module $\Psi\tcfr{a,d}{b}{c}$. 
\end{theorem}

Now we return to the proof of our main result, Theorem~\ref{mainthm}:

\begin{proof}[Proof of Theorem \ref{mainthm}]
    We saw that any web in $\mathcal{W}$ comes from stretching a primitive web in $\Gr(k,2k)$, and is thus a cluster variable (Theorem~\ref{thm:web-cluster}, Corollary~\ref{cor:allGrassm}). Any web in $\mathcal{W}$ is also related via stretching with the functions $\tcfr{a,d}{b}{c}$. For the webs $\tcfr{a,d}{b}{c}$ we know the corresponding modules (as defined in Section~\ref{sec:wm}) by Theorem~\ref{thm:ses}.
    
    But by Lemma~\ref{stretchinglemma}, we understand how stretching of functions corresponds to stretching of modules: if $\Psi(W)$ is the module corresponding to $W$, this holds for any web obtained from stretching $W$. Thus for all $W \in \mathcal{W}$, $\Psi(W)$ is the module corresponding to $W$.

    This gives the bijection between $\mathcal{W}$ and $\mathcal{M}$.
\end{proof}

\bibliographystyle{alpha}
\bibliography{references}

\end{document}